\documentclass[12pt]{article} 

\usepackage[utf8]{inputenc}
\usepackage[english]{babel}
\usepackage[T1]{fontenc}
\usepackage{microtype}

\usepackage[margin=1in]{geometry}

\usepackage{amsmath, amssymb, amsfonts, amsthm, mathtools, bbm}
\usepackage{tikz}
\usetikzlibrary{arrows.meta, shapes, arrows, positioning}
\usepackage{enumitem}
\usepackage{comment}

\newtheorem{theorem}{Theorem}[section]

\newtheorem{lemma}[theorem]{Lemma}
\theoremstyle{definition}
\newtheorem{definition}[theorem]{Definition}
\newtheorem{remark}[theorem]{Remark}

\newcommand{\mN}{\mathbb{N}}

\newcommand{\Pb}{\mathsf{P}}
\newcommand{\E}{\mathsf{E}}
\newcommand{\1}{\mathbf{1}} 
\newcommand{\ee}{\mathrm{e}}

\let\oldmarginpar\marginpar
\renewcommand{\marginpar}[1]{\oldmarginpar{\scriptsize\texttt{\color{blue}{#1}}}}

\usepackage{titling}
\pretitle{\begin{center}\Large\bfseries}
\posttitle{\end{center}}

\title{Explosion and implosion of birth-and-death continuous-time random walks}

\makeatletter
\let\@fnsymbol\@arabic
\makeatother

\author{
Andrey Pilipenko\thanks{\texttt{pilipenko.ay@gmail.com}, Universit\'e de Gen\`eve, 
Section de math\'ematiques, Switzerland; Institute of Mathematics of NAS of Ukraine.}
\and 
Vadym Tkachenko\thanks{\texttt{vtkachenko@kse.org.ua}, Institute of Mathematics of NAS of Ukraine; Kyiv School of Economics.
}
}

\date{}

\usepackage{titlesec}
\titleformat{\section}
  {\normalfont\large\bfseries\centering}{\thesection.}{0.5em}{}
\titleformat{\subsection}
  {\normalfont\normalsize\bfseries\centering}{\thesubsection.}{0.5em}{}

\titleformat{\subsubsection}[runin] 
  {\normalfont\normalsize\bfseries} 
  {\thesubsubsection.} 
  {0.5em} 
  {} 
  [.] 

\usepackage[colorlinks=true,
    linkcolor=blue!60!black,     
    citecolor=red!70!black,      
    urlcolor=green!50!black      
]{hyperref}

\usepackage{cleveref}

\begin{document}
\renewcommand{\qedsymbol}{$\blacksquare$}

\maketitle
\vspace{-3em}
\begin{abstract}
    We provide necessary and sufficient conditions for explosion and implosion of birth-and-death (non-Markov) continuous-time random walks. In other words, we obtain conditions for $\infty$ to be accessible and for it to be an entrance point. We derive the analytical regularity criteria in terms of the appropriate scale function and the speed measure, which involve transition probabilities and the Laplace transform of the waiting times. We show that these criteria closely resemble classical ones for diffusions and Markov birth-and-death processes. We then calculate explicit conditions of regularity for semi-Markov processes with waiting times that have (a) finite first moments; (b)~regularly varying tails (in particular, $\alpha$-stable distribution).
    
    \phantom{}
    

\end{abstract}

\section{Introduction}\label{introduction}

Semi-Markov processes, commonly known as continuous-time random walks, were introduced by L\'evy \cite{levy1954Processus} and Smith \cite{Smith1955Regenerative} as a generalization of Markov processes on a countable state space. Markov property of the latter dictates the distribution of waiting times to be exponential, which is not the case for many observed processes (population processes, queueing networks, microbiological processes).

Questions about regularity arise naturally when modeling physical or natural phenomena. A core issue is whether a process, while evolving continuously in time, can travel through infinitely many states in a finite duration. This phenomenon, known as explosion, must be ruled out for processes that model clearly non-explosive systems. Another classical question is that of the entrance boundary, or implosion: can a process be initiated from its boundary that is a not part of the state space. We provide all needed definitions and some preliminaries in the Section \ref{definitions}. We proceed with a short historical overview of these problems.

Boundary classification of one dimensional diffusion processes begins with Feller's 1952 paper \cite{Feller52ParabDiffEq}. In Section 11 of this classical work, he introduces four types of boundaries depending on the behavior of the diffusion: regular, exit, entry, and natural. His approach lies in the semigroup method. To express the boundary conditions, Feller uses what is later to be called the scale function. In the follow-up 1957 paper \cite{Feller1957Generalized} he introduces the concept of the speed measure and establishes a canonical infinitesimal generator in terms of it and the scale function. These notions play a central role in the current paper. Boundary classification for diffusions was advanced and completed using the sample-path approach in the book by It{\^o} and McKean \cite{ItoMcKean2012(1965)Diffusion}. For more details, we refer to a comprehensive overview by Peskir~\cite{Peskir2015Boundary}.

Moving to Markov processes on a countable state space, Feller discusses birth-and-death processes in a 1959 paper \cite{Feller1959BirthDeath}, using probabilistic approach to boundary problems. Dynkin and Yushkevich \cite{D_Y} provide analytical conditions for explosion of such processes in terms of transition probabilities and the expectation of waiting times. They construct the appropriate scale function and speed measure such that the regularity condition becomes as in the case of diffusion. Their approach is of a great importance to the current paper. 

Boundary behavior for general Markov processes is challenging. For time-homogeneous Markov chains, the problem is resolved in \cite{chung2012markov, Feller1957Boundaries}. Results also exist for continuous-state branching processes with competition \cite{Foucart2019Branch_proc_with_comp_dual_and_refl_at_infty}.

For semi-Markov processes on a countable state space, Feller \cite{Feller1964SemiMarkov} derives an analogue of Kolmogorov's backward equations and uses the Laplace transform to obtain a minimal solution. He shows that the boundary can be reached in finite time if and only if a certain problem admits a minimal substochastic solution. However, this condition is difficult to verify in practice. Gerrard \cite{gerrard1983regularity} uses a similar approach to obtain explosion conditions, but they remain implicit.

We take a different approach and focus specifically on birth-and-death processes. This allows us to derive natural and explicit analytical conditions in terms of the appropriate scale function and speed measure.

\section{Definitions and formulations}\label{definitions}
Consider a birth-and-death Markov chain $(X_k)_{k \ge 0}$ on the state space $\{0,1,2,\dots\}$ with transition probabilities $p_{i,i+1} = p_i$, $p_{i, i-1} = q_i$ such that $p_{0,1} = 1$ and $p_i$, $q_i$ are positive with $p_i + q_i = 1$ for $i \ge 1$. It can be represented with the following diagram:
    
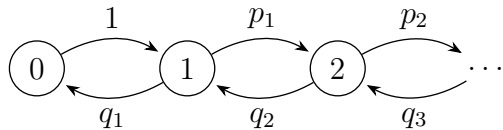
\begin{center}
        \begin{tikzpicture}[-Stealth, >=stealth, shorten >=2pt, line width=0.5 pt, node distance=2 cm]

\node[circle,  draw]  (zero)  {0};
\node[circle, draw] (one)  [right of=zero] 
  {1};
  \node[circle, draw] (two) [right of=one]{2};
  \node[draw=none] (three) [right of=two] {\dots};
\path (zero) edge [bend left] node [above ] {$1$} (one) ;
\path (one) edge [bend left] node [below] {$q_1$} (zero);
\path (one) edge [bend left] node [above ] {$p_1$} (two);
\path (two) edge [bend left] node [below] {$q_2$} (one);
\path (two) edge [bend left] node [above ] {$p_2$} (three) ;
\path (three) edge [bend left] node [below] {$q_3$} (two);
\end{tikzpicture}
    \end{center}
Let $\{\tau_i^+, \tau_i^- \mid i \ge 0\}$ be positive random variables representing right and left waiting times, respectively. For each $\tau_i^+$ and $\tau_i^-$, consider sequences $\{\tau_{i,k}^+\}_{k\ge0}$ and $\{\tau_{i,k}^-\}_{k\ge0}$ of their independent copies. Suppose all these sequences and the Markov chain $(X_k)$ are independent of each other. Put $T_{0} = 0$ and define random moments of jumps recurrently: 
\begin{equation*}
    T_{k+1} = T_k + \tau^k, \quad k \ge 0,
\end{equation*}
where
 \begin{align*}
    \tau^k = \begin{cases}
        \tau_{i, k}^+ \quad \text{if }  X_k=i \text{ and } X_{k+1}=i+1, \\
        \tau_{i, k}^- \quad \text{if } X_k=i\text{ and } X_{k+1}=i-1.
    \end{cases}         
    \end{align*}
    If $X_k = i$, then the distribution of $\tau^k$ will be denoted by $\tau_i$.
\begin{definition}\label{def of X}
The process $X(t)$, $t \ge 0$, defined by the formula
    \begin{equation*}
        X(t) = X_k, \quad t \in [T_{k}, T_{k+1}), k \ge 0,
    \end{equation*}
 is called a \emph{continuous-time random walk} 
with the embedded Markov chain $(X_k)$ and waiting times $\{\tau_i^+, \tau_i^- \}$.
\end{definition}

\begin{definition}\label{def of moment of exp}
Denote $\sigma_n = \inf\{t \ge 0 \mid X(t) = n \}$, i.e., the first moment when $X$ hits the state $n \in \{0, 1, 2, \dots\}$. Then $\sigma_\infty =\lim_{n \to \infty} \sigma_n$ is called the\emph{moment of explosion} of the process $X$. If $\sigma_\infty < \infty$, it is said that $X$ \emph{explodes} to $\infty$.
\end{definition}

\begin{remark}
By the moment of time $\sigma_\infty$, the process has traveled through infinitely many states. Processes having $\sigma_\infty= \infty$ a.s.~are called \emph{regular}. Otherwise, \emph{non-regular}.
\end{remark}
Our first goal is to find analytical conditions of regularity for $X$. 
Our second goal is to examine a process
that ``starts at infinity'' and establish when it can ``break away from infinity'' in a finite time. By $\Pb_i$ denote the conditional probability under $X(0) = i$.
\begin{definition}
    The process \emph{implodes} from infinity if 
\begin{equation*}
\lim_{C\to+\infty}\sup_i\Pb_i(\sigma_0>C)=0.
    \end{equation*}
\end{definition}
We provide a more suitable version of this definition in Section \ref{alt construction}, using an alternative construction of $X$. An equivalent definition for the Markov case can be found in \cite[p.~313]{MenshikovPopovWade2016RandomWalks}. 

\section{Main results}
In this section, we formulate two main theorems of this paper and compare them to known results for diffusion processes. Let $X$ be a continuous-time random walk
from Definition~\ref{def of X}. Introduce the \emph{scale}
\begin{equation}\label{eqn: delta}
    \delta_k =  \frac{q_1 \dots q_k}{p_1 \dots p_k}, \quad k \ge 0, 
\end{equation}
and the \emph{speed measure}
\begin{equation}\label{eqn: nu}
    \nu_i = (1-\E \ee^{-\tau_i}) \frac{p_{1} \dots p_{i-1}}{q_1 \dots q_i}, \quad i \ge 0,
\end{equation}
where $\E \ee^{-\tau_i} = p_i\E \ee^{-\tau_i^+}+q_i\E \ee^{-\tau_i^-}$, and the product over an empty set is equal to $1$, hence
\begin{equation*}
    \delta_0 = 1, \quad \nu_0 =  1-\E \ee^{-\tau_0^+}, \quad \nu_1 = (1-\E \ee^{-\tau_1})\frac{1}{q_1}.
\end{equation*}
\begin{theorem}[Explosion condition]\label{explosion theorem} 
For every initial distribution of $X$, the following alternative holds:
\begin{align}
     \mathsf P \{X \text{ explodes to }\infty \} = 1 \quad \Longleftrightarrow \quad \sum_{k = 0}^\infty  \left(\sum_{i=0}^{k} \nu_i \right) \delta_k < \infty, \label{eqn: explosion condition}\\
     \mathsf P \{X \text{ explodes to }\infty \} = 0 \quad  \Longleftrightarrow \quad \sum_{k = 0}^\infty  \left(\sum_{i=0}^{k} \nu_i \right) \delta_k = \infty. \label{eqn: no explosion condition}
\end{align}
\end{theorem}
\begin{theorem}[Implosion condition]\label{implosion theorem}
    For every initial distribution of $X$, the following alternative holds:
    \begin{align}
    &\Pb \{X \text{ implodes from }\infty\} = 1 \quad \Longleftrightarrow \quad \sum_{k = 0}^\infty \left(\sum_{i=k+1}^\infty \nu_i \right)\delta_k  < \infty \quad\text{and} \quad \sum_{k=0}^\infty \delta_k = \infty, \label{eqn: implosion condition} \\
    &\Pb \{X \text{ implodes from }\infty\} = 0 \quad \Longleftrightarrow \quad \sum_{k = 0}^\infty  \left( \sum_{i=k+1}^\infty \nu_i \right) \delta_k= \infty  \quad \text{ or } \quad \sum_{k=0}^\infty \delta_k < \infty. \label{eqn: no implosion condition}
\end{align}
\end{theorem}
Proofs can be found in Section \ref{proof}.
\begin{remark}
    If we replace the Laplace transform $\E \ee^{-\tau_i}$ with $\E \ee^{-\lambda \tau_i}$ for some $\lambda > 0$, the results will stay the same.
\end{remark}

\subsection{Comparison with known results for diffusions}
Let us compare the conditions for explosion and implosion with known results for one-dimensional diffusions (see \cite[XV, \S 6]{karlin1981second} or \cite[V, \S 11]{Ito_Stoch_Proc}). Assume that $X$ is a one dimensional diffusion with the scale function $S$ and the speed measure $M$, where $S, M \colon \mathbb{R} \to \mathbb{R}$ are strictly increasing, with $S$ continuous. Then the boundary $+\infty$ is \emph{attainable} in finite time with positive probability if and only if
\begin{equation}\label{eqn: explosion diffusion}
    \int_{x_0}^\infty M [x_0, x]\, \mathrm{d}S(x) < \infty,
\end{equation}
where $M[x_0, x] = M(x) - M(x_0)$, and $x_0$ is arbitrary (convergence of the above integral is independent of $x_0$).

The boundary $+\infty$ is an \emph{entrance point} for a strong Markov process $X(t)$, $t > 0$, if $\lim_{t \to 0+} X(t) = +\infty$. In terms of the scale function and the speed measure of a diffusion, this happens if and only if 
\begin{equation}\label{eqn: implosion diffusion}
    \int_{x_0}^\infty M [x, \infty]\, \mathrm{d}S(x) < \infty \quad \text{and} \quad S(\infty) = \infty.
\end{equation}

Define the scale function and the speed measure using $\delta_k$, $\nu_i$ from \eqref{eqn: delta}, \eqref{eqn: nu} as
\begin{equation*}
    S(n) = 1 + \sum_{k = 1}^{n-1} \delta_k, \quad M(n) = \sum_{i=1}^n \nu_i, \quad n \in \mathbb{N}.
\end{equation*}
Then \eqref{eqn: explosion diffusion} is equivalent to our explosion condition~\eqref{eqn: explosion condition}, and \eqref{eqn: implosion diffusion} is equivalent to our implosion condition \eqref{eqn: implosion condition} if we replace the integrals by discreet sums.

\section{Examples}
In this section, we apply Theorems \ref{explosion theorem} and~\ref{implosion theorem} to specific types of processes. Section \ref{Markov case} covers the Markov case. We show that the explosion condition \eqref{eqn: explosion condition} coincides with a well-known regularity condition. We then provide new results for non-Markov processes. Section \ref{waiting times with first moments} covers processes that have waiting times with finite first moments. Section \ref{reg var tails} deals with the case of heavy-tailed waiting times.

The following will be useful throughout this section. Denote $\Phi_i = 1 - \E \ee^{-\tau_i}$, and rewrite the series from the explosion condition changing the order of summation:
\begin{equation}\label{eqn: explosion change of order}
    \sum_{k = 0}^\infty  \left(\sum_{i=0}^{k} \nu_i \right) \delta_k = \sum_{i=0}^\infty \sum_{k=i}^\infty \nu_i \delta_k = \sum_{i=0}^\infty \sum_{k=i}^\infty \Phi_i \frac{p_{1} \dots p_{i-1}}{q_1 \dots q_i} \frac{q_1 \dots q_k}{p_1 \dots p_k} = \sum_{i=0}^\infty \Phi_i \sum_{k=i}^\infty  \frac{q_{i+1} \dots q_k}{p_i \dots p_k}.
\end{equation}
 Analogously, rewrite the series from the implosion condition:
\begin{equation}\label{eqn: implosion change of order}
    \sum_{k = 0}^\infty  \left(\sum_{i=k+1}^{\infty} \nu_i \right) \delta_k = \sum_{i=1}^\infty \sum_{k=0}^{i-1} \nu_i \delta_k = \sum_{i=1}^\infty \sum_{k=0}^{i-1} \Phi_i \frac{p_{1} \dots p_{i-1}}{q_1 \dots q_i} \frac{q_1 \dots q_k}{p_1 \dots p_k} = \sum_{i=1}^\infty \Phi_i \sum_{k=0}^{i-1}  \frac{p_{k+1} \dots p_{i-1}}{q_{k+1} \dots q_i}.
\end{equation}
In these formulas, we agree that $q_{i+1} \dots q_k = 1$ if $i = k$, and $p_{k+1} \dots p_{i-1} = 1$ if  $k = i-1$. Note that convergence of either series implies 
\begin{equation}\label{eqn: Phi to 0}
    \lim_{i \to \infty} \Phi_i = 0.
\end{equation}
\subsection{Markov case}\label{Markov case}
Suppose $X$ is a birth-and-death Markov processes. In this case, the waiting times $\tau_i^+$ and $\tau_i^-$ are equally exponentially distributed with some rate parameters $a_i>0$. The Laplace transform is $\mathsf{E} \mathrm{e}^{-\lambda \tau_i} = \frac{a_i}{\lambda + a_i}$, and hence $\Phi_i =1 - \mathsf{E} \mathrm{e}^{- \tau_i} = \frac{1}{1 + a_i}$. 

Consider the explosion condition \eqref{eqn: explosion condition}. Then using \eqref{eqn: explosion change of order}, \eqref{eqn: Phi to 0}, and the fact that $\frac{1}{1+a_i} \sim \frac{1}{a_i}$ as~$a_i \to \infty$, we arrive at 
\begin{equation}\label{eqn: Mark. case exp. cond.}
    \sum_{i = 0}^{\infty} \frac{1}{ a_i} \sum_{k = i}^\infty \frac{q_{i+1} \dots q_k}{p_i \dots p_k} < \infty.
\end{equation}
This is a well-known explosion condition for Markov birth-and-death processes (\cite[IV,~\S~5]{D_Y}).

Now consider a homogeneous process with $p_i \equiv p$ and $q_i \equiv q = 1 - p$. If $p > q$, the series in \eqref{eqn: Mark. case exp. cond.} simplifies as follows:
\begin{equation*}
    \sum_{i = 0}^{\infty} \frac{1}{ a_i} \sum_{k = i}^\infty \frac{q_{i+1} \dots q_k}{p_i \dots p_k} = \sum_{i = 0}^{\infty} \frac{1}{ a_i p} \sum_{l = 0}^\infty \left(\frac{q}{p}\right)^l = \sum_{i = 0}^{\infty} \frac{1}{ a_i p} \frac{1}{1 - \frac{q}{p}} = \frac{1}{p - q} \sum_{i = 0}^{\infty} \frac{1}{ a_i}.
\end{equation*}
Thus, the explosion condition for a homogeneous birth-and-death Markov process is
\begin{equation}\label{eqn: Mark. case exp. cond. homog. case}
    \sum_{i = 0}^{\infty} \frac{1}{ a_i} < \infty \quad \text{and} \quad p > q.
\end{equation}
It coincides with the explosion condition for the Yule process of pure birth (see \cite[XIV, \S8]{Feller_Introduction}).
The condition makes intuitive sense, because the average number of visits of each state by the Markov chain $(X_k)$ equals $1/(p-q)$.

Consider implosion condition \eqref{eqn: implosion condition}. Analogously to explosion, using \eqref{eqn: implosion change of order} and \eqref{eqn: Phi to 0}, we get the implosion condition for a birth-and-death Markov process:
\begin{equation}\label{eqn: Mark. case imp. cond.}
    \sum_{i = 1}^\infty \frac{1}{a_i}\sum_{k=0}^{i-1}  \frac{p_{k+1} \dots p_{i-1}}{q_{k+1} \dots q_i} < \infty \quad \text{and} \quad \sum_{k=0}^\infty \delta_k = \infty.
\end{equation}

Now consider a homogeneous process
with $p_i \equiv p$ and $q_i \equiv q = 1 - p$. If $p < q$, we get
\begin{equation*}
    \sum_{k=0}^{i-1}  \frac{p_{k+1} \dots p_{i-1}}{q_{k+1} \dots q_i} = \sum_{k=0}^{i-1} \left(\frac{p}{q}\right)^{i-k} \frac{1}{p} = \frac{1 - \Big(\frac{p}{q}\Big)^i}{q-p}.
\end{equation*}
For a symmetric walk with $p = q = 1/2$, we get
\begin{equation*}
    \sum_{k=0}^{i-1}  \frac{p_{k+1} \dots p_{i-1}}{q_{k+1} \dots q_i} = 2 i.
\end{equation*}
Therefore, the condition \eqref{eqn: Mark. case imp. cond.} holds if and only if one of the following is true:
\begin{align}
    \sum_{i = 1}^\infty \frac{1}{a_i} < \infty \quad \text{and} \quad p < q, \label{eqn: Mark. case imp. cond. homog. case p<q} \\
    \sum_{i = 1}^\infty \frac{i}{a_i} < \infty \quad \text{and} \quad p = q. \label{eqn: Mark. case imp. cond. case p=q}
\end{align}

\subsection{Waiting times with finite first moments}\label{waiting times with first moments}
Suppose $\tau_i^\pm \overset{d}{=} \tau^\pm /a_i^\pm$, where $\tau^\pm$ are positive random variables, and $a_i^\pm$ are positive numbers with $\lim_{i \to \infty} a_i^\pm = \infty$. 
Assume that $\E \tau^\pm < \infty$. It is well known that the asymptotics of the Laplace transform in the neighborhood of zero is $1 - \E \ee^{-\lambda \tau} \sim \lambda \E \tau$ as $\lambda \to 0+$. Therefore,
\begin{align*}
    1 - \E \ee^{-\tau_i} &= 1-p_i\E \ee^{-\tau_i^+}-q_i\E \ee^{-\tau_i^-} 
    = 1-p_i\E \ee^{-\tau^+ /a_i^+}-q_i\E \ee^{-\tau^- /a_i^-}\\ 
    &= p_i\left(1 -\E \ee^{-\tau^+ /a_i^+}\right)+ q_i\left(1-\E \ee^{-\tau^- /a_i^-}\right)\sim \frac{p_i}{a_i^+} \E \tau^+ + \frac{q_i}{a_i^-} \E \tau^-, \quad i \to \infty.
\end{align*}
In particular, if $\E\tau^+ = \E\tau^-$ and $a_i^+ = a_i^-$, then the conditions for explosion and implosion in this case are identical to conditions \eqref{eqn: Mark. case exp. cond.}--\eqref{eqn: Mark. case imp. cond. case p=q} in the Markov case from Section \ref{Markov case}.

\subsection{Waiting times with regularly varying tails}\label{reg var tails}

Suppose $\tau_i^\pm \overset{d}{=} \tau^\pm /a_i^\pm$, where $\tau^\pm$ are positive random variables with distribution functions~$F^\pm$, and $a_i^\pm$ are positive numbers with $\lim_{i \to \infty} a_i^\pm = \infty$. Suppose that $1-F^\pm$ varies regularly at $\infty$ with exponent $-\alpha^\pm$, where $0 < \alpha^\pm < 1$. That is,
\begin{equation*}
    1 - F^\pm(x) = x^{-\alpha^\pm} l^\pm(x), \quad x \to +\infty,
\end{equation*}
where $l^\pm$ are slowly varying functions at $\infty$ (see \cite[I, \S4]{bingham1989regular}). Then from a Tauberian-Abelian Theorem (see \cite[XIII, \S5, Thm.~1]{Feller_Introduction}), we get the following asymptotics of the Laplace transform:
\begin{equation*}
    1 - \varphi^\pm(\lambda) \sim \Gamma(1-\alpha^\pm) (1 - F^\pm(1/\lambda) ), \quad \lambda \to {0+}.
\end{equation*}
The explosion condition \eqref{eqn: explosion condition} becomes
\begin{equation}\label{eqn: reg.varying waiting times exp. condition}
    \sum_{i = 0}^{\infty} \left(1 - p_iF^+(a_i^+) - q_iF^-(a_i^-) \right) \sum_{k = i}^\infty \frac{q_{i+1} \dots q_k}{p_i \dots p_k} < \infty.
\end{equation}
The implosion condition \eqref{eqn: implosion condition} becomes
\begin{equation}\label{eqn: reg var imp. cond.}
    \sum_{i = 1}^\infty \left(1 - p_iF^+(a_i^+) - q_iF^-(a_i^-) \right)\sum_{k=0}^{i-1}  \frac{p_{k+1} \dots p_{i-1}}{q_{k+1} \dots q_i} < \infty \quad \text{and} \quad \sum_{k=0}^\infty \delta_k = \infty.
\end{equation}
\subsection{Stable waiting times}
An interesting special case of Section \ref{reg var tails} is the case where $\tau^+$ and $\tau^-$ are equally distributed positive random variables with $\alpha$-stable distribution, $\alpha\in(0,1)$, and $a_i^+ = a_i^- = a_i$. Then the Laplace transform is
\begin{equation*}
    \E \ee^{-\tau_i} = \exp\{-c a_i^{-\alpha}\} , \quad i \ge 0,
\end{equation*}
for some constant $c > 0$ (see \cite[XIII, \S6]{Feller_Introduction}). Therefore, the following equivalence holds:
\begin{equation*}
    1 - \E \ee^{-\tau_i} = 1 - \exp\{-c a_i^{-\alpha}\} \sim c a_i^{-\alpha}, \quad a_i \to \infty.
\end{equation*}
The explosion condition \eqref{eqn: explosion condition} becomes
\begin{equation}\label{eqn: heavy tailed waiting times exp. condition}
    \sum_{i = 0}^{\infty} \frac{1}{a_i^\alpha} \sum_{k = i}^\infty \frac{q_{i+1} \dots q_k}{p_i \dots p_k} < \infty.
\end{equation}
The implosion condition \eqref{eqn: implosion condition} becomes
\begin{equation}\label{eqn: heavy tailed waiting times imp. condition}
    \sum_{i = 1}^{\infty} \frac{1}{a_i^\alpha} \sum_{k = 0}^{i-1}\frac{p_{k+1} \dots p_{i-1}}{q_{k+1} \dots q_i} < \infty \quad \text{and} \quad \sum_{k=0}^\infty \delta_k = \infty.
\end{equation}
If $p_i \equiv p$ and $q_i \equiv q$, condition \eqref{eqn: heavy tailed waiting times exp. condition} becomes
\begin{equation}\label{eqn: heavy tailed waiting times exp. condition homog. case}
\sum_{i = 0}^{\infty} \frac{1}{a_i^\alpha} < \infty \quad \text{and} \quad p > q.
\end{equation}
Condition \eqref{eqn: heavy tailed waiting times imp. condition} becomes one of the following:
\begin{align}
     & \sum_{i = 0}^{\infty} \frac{1}{a_i^\alpha} < \infty \quad \text{and} \quad p < q, \\
    &\sum_{i = 0}^{\infty} \frac{i}{a_i^\alpha} < \infty \quad \text{and} \quad p = q.\label{eqn: heavy tailed waiting times imp. condition homog. case}
\end{align}

\begin{remark} Conditions \eqref{eqn: heavy tailed waiting times exp. condition}--\eqref{eqn: heavy tailed waiting times imp. condition homog. case} coincide with conditions \eqref{eqn: Mark. case exp. cond.}--\eqref{eqn: Mark. case imp. cond. case p=q} from the Markov case if we formally set $\alpha = 1$.
\end{remark}

\section{Proof}\label{proof}
In this section, we prove Theorems \ref{explosion theorem} and \ref{implosion theorem}. Sections \ref{alt construction} and \ref{transience, recc and scale} contain some useful auxiliary constructions and results. Section \ref{preliminaries} introduces key objects, tools, and lemmas. The main arguments are presented in Sections \ref{regularity proof} and \ref{implosion proof}, which establish the explosion and implosion conditions, respectively.

\subsection{Alternative construction}\label{alt construction}

We give an alternative construction of the process $X$ in order to define implosion in a more suitable manner. This construction will also be useful for investigating explosion. Let $\{Y_n\}_{n \ge 1}$ be a sequence of independent processes from Definition \ref{def of X} with $Y_n(0) = n$ for $n \ge 1$. Define stopping times
\begin{equation*}
    \theta_n = \inf\{t \ge 0 \mid Y_n(t) = n-1 \}, \quad n \ge 1.
\end{equation*}
An example of trajectories of such processes is given in Figure \ref{fig: trajectories}, each trajectory $Y_n$ is drawn up to the respective stopping time $\theta_n$.
\begin{figure}[bp]
    \centering
    \begin{tikzpicture}

    \draw[->] (0,0) -- (12,0) node[below] {$t$}; 
    \draw[->] (0,0) -- (0,5) node[left] {};    
    
    \node[left] at (0,0) {$0$};
    \node[left] at (0,1) {$1$};
    \node[left] at (0,2) {$2$};
    \node[left] at (0,3) {$3$};
    \node[left] at (0,4) {$4$};

    \draw (2,0) node[below] {$\theta_1$};
    \draw (3,0) node[below] {$\theta_2$};
    \draw (5,0) node[below] {$\theta_3$};
    
    \draw[thick, blue] (0,1) -- (2,0); 
    \fill[black] (0,1) circle (1pt); 
    \fill[black] (2,0) circle (1pt); 
    
    \draw[thick, red] (0,2) -- (1.5,3) -- (2.1,2) -- (3,1); 
    \fill[black] (0,2) circle (1pt); 
    \fill[black] (1.5,3) circle (1pt); 
    \fill[black] (2.1,2) circle (1pt); 
    \fill[black] (3,1) circle (1pt); 
    
    \draw[thick, green] (0,3) -- (0.6,4) -- (2,3) -- (3.1,4) -- (3.8,3) -- (5,2); 
    \fill[black] (0,3) circle (1pt); 
    \fill[black] (0.6,4) circle (1pt); 
    \fill[black] (2,3) circle (1pt); 
    \fill[black] (3.1,4) circle (1pt); 
    \fill[black] (3.8,3) circle (1pt); 
    \fill[black] (5,2) circle (1pt); 
    
    \node[above right] at (1,0.3) {$Y_1(t)$};
    \node[above right] at (2.5,1.3) {$Y_2(t)$};
    \node[above right] at (4.4,2.3) {$Y_3(t)$};
    
    \draw[dashed] (3,0) -- (3,1);
    \draw[dashed] (5,0) -- (5,2);
    \draw[dashed] (0,1) -- (3,1);
    \draw[dashed] (0,2) -- (5,2);
    \end{tikzpicture}
    \caption{Trajectories of $Y_1$, $Y_2$ and $Y_3$ up to their stopping times.}
    \label{fig: trajectories}
\end{figure}
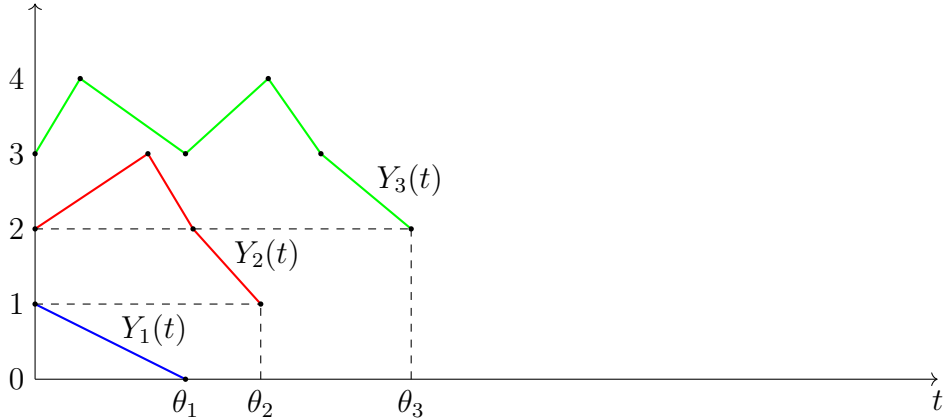
For a better visual representation, we modify the trajectories to be continuous piece-wise linear instead of piece-wise constant. We then construct a sequence of processes $\{Z_n\}_{n \ge 1}$ by shifting the trajectories of process $Y_n$ and joining them one to another as it is shown in Figure \ref{fig: trajectory of Z3}. 
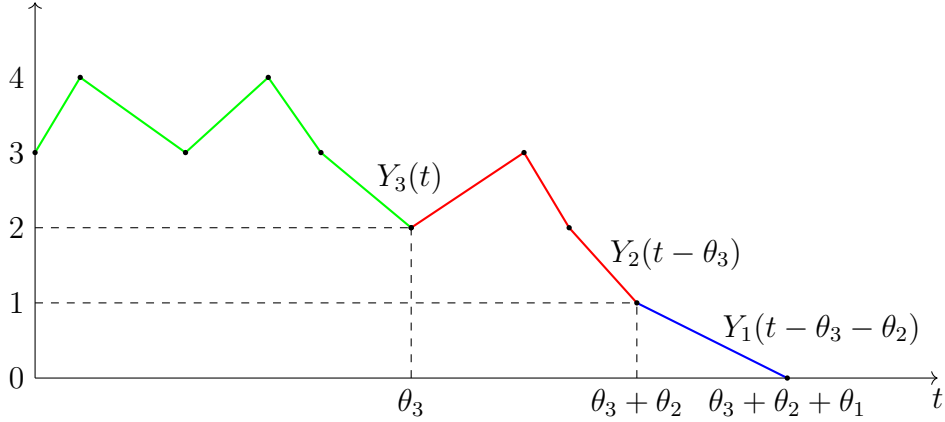
\begin{figure}[tp]
    \centering
        \begin{tikzpicture}

    \draw[->] (0,0) -- (12,0) node[below] {$t$}; 
    \draw[->] (0,0) -- (0,5) node[left] {};    
    
    \node[left] at (0,0) {$0$};
    \node[left] at (0,1) {$1$};
    \node[left] at (0,2) {$2$};
    \node[left] at (0,3) {$3$};
    \node[left] at (0,4) {$4$};

    \draw (10,0) node[below] {$\theta_3 + \theta_2 + \theta_1$};
    \draw (8,0) node[below] {$\theta_3 + \theta_2$};
    \draw (5,0) node[below] {$\theta_3$};
    
    \draw[thick, blue] (8,1) -- (10,0); 
    
    \fill[black] (10,0) circle (1pt); 
    
    \draw[thick, red] (5,2) -- (6.5,3) -- (7.1,2) -- (8,1); 
    \fill[black] (5,2) circle (1pt); 
    \fill[black] (6.5,3) circle (1pt); 
    \fill[black] (7.1,2) circle (1pt); 
    \fill[black] (8,1) circle (1pt); 
    
    \draw[thick, green] (0,3) -- (0.6,4) -- (2,3) -- (3.1,4) -- (3.8,3) -- (5,2); 
    \fill[black] (0,3) circle (1pt); 
    \fill[black] (0.6,4) circle (1pt); 
    \fill[black] (2,3) circle (1pt); 
    \fill[black] (3.1,4) circle (1pt); 
    \fill[black] (3.8,3) circle (1pt); 
    \fill[black] (5,2) circle (1pt); 
    
    \node[above right] at (9,0.3) {$Y_1(t-\theta_3 - \theta_2)$};
    \node[above right] at (7.5,1.3) {$Y_2(t-\theta_3)$};
    \node[above right] at (4.4,2.3) {$Y_3(t)$};
    
    \draw[dashed] (8,0) -- (8,1);
    \draw[dashed] (5,0) -- (5,2);
    \draw[dashed] (0,1) -- (8,1);
    \draw[dashed] (0,2) -- (5,2);
    \end{tikzpicture}
    \caption{Trajectory of $Z_3$ obtained by shifting and joining $Y_1$, $Y_2$ and $Y_3$.}
    \label{fig: trajectory of Z3}
\end{figure}
In general, $Z_n$ is a continuous-time random walk starting at $n$ defined by the formula
\begin{equation}\label{eqn: Z_n}
    Z_n(t) = \begin{cases}
        Y_n(t), & \quad t \in [0, \ \theta_n), \\
        Y_{n-1}(t-\theta_n), & \quad t \in [\theta_n, \ \theta_n + \theta_{n-1}),\\
        Y_{n-2} (t - \theta_n - \theta_{n-1}), &\quad t \in [\theta_n + \theta_{n-1}, \ \theta_n + \theta_{n-1} + \theta_{n-2}),\\
        \dots\\
        Y_1 (t - \theta_n - \dots - \theta_2),& \quad t \in [\theta_n + \dots + \theta_2, \ \theta_n + \dots + \theta_2 + \theta_1), \\
        Y_0(t - \theta_n - \dots - \theta_2 - \theta_1), &\quad t \in [\theta_n + \dots + \theta_2 + \theta_1, \ \infty ).
    \end{cases}
\end{equation} 
Then $\rho_\infty = \sum_{k=1}^\infty \theta_k$ represents the first time of hitting $0$ when starting at infinity. Now we can define implosion in a more suitable way.
\begin{definition}\label{alt def of implosion}
    We say that the process $X$ \emph{implodes} from infinity if $\rho_\infty = \sum_{k=1}^\infty \theta_k < \infty $. 
\end{definition}

\begin{remark}
    Since $\theta_k$ are independent random variables, the event $\{ \rho_\infty < \infty\}$ happens with probability either $0$ or $1$ by Kolmogorov's zero-one law.
\end{remark}

\begin{remark}
    Consider the compactified space $\mathbb{N}_0\cup\{\infty\}$ with the metric $d(i, j) = |\arctan(i) - \arctan(j)|$. It is not difficult to show that if implosion happens, the sequence $\{Z_n\}_{n \ge 0}$ defined in \eqref{eqn: Z_n} converges in the space $D([0, \infty), \mathbb{N}_0\cup\{\infty\})$ of c\`adl\`ag functions in Skorokhod topology almost surely. 
\end{remark}

\begin{remark}\label{alt construction for mom of exp}
The alternative construction can be used to define the moment of explosion, too. This turns out to be very useful, because we immediately get that explosion also happens with probability either $0$, or $1$. Again, let $\{Y_n\}_{n \ge 0}$ be a sequence of independent processes from Definition \ref{def of X} such that $Y_n(0) = n$ for $n \ge 0$. Define stopping times
\begin{equation*}
    \eta_n = \inf\{t \ge 0 \mid Y_n(t) = n+1 \}, \quad n \ge 1.
\end{equation*}
Then the connection with $\sigma_\infty$ defined in \ref{def of moment of exp} is the following: conditioned on $X(0) = i$,
\begin{equation*}
    \sigma_\infty = \lim_{n \to \infty}\sigma_n = \lim_{n \to \infty} \sum_{k = i + 1}^n (\sigma_k - \sigma_{k-1}) \overset{d}{=} \lim_{n \to \infty} \sum_{k = i + 1}^n \eta_k = \sum_{k = i + 1}^\infty \eta_k.
\end{equation*}
\end{remark}

\subsection{Transience, recurrence and the canonical scale}\label{transience, recc and scale}
Recall that a Markov chain $(X_k)$ is called \emph{transient} if, starting from any state, there is a positive probability of never returning to it. The Markov chain is called \emph{recurrent} if it returns to each initial state with probability one. It is straightforward that explosion of the process $X$ may occur only if its embedded Markov chain $(X_k)$ is transient, and implosion may occur only if it is recurrent. To formulate analytical conditions of transience and recurrence, let us introduce the \emph{canonical scale} (see \cite[p.~157]{D_Y}). Consider a set $\{u_0, u_1, u_2, \dots \}$ such that  $u_0 = 0$, $u_1 = 1$,
\begin{equation}\label{eqn: scale}
     u_i = 1 + \delta_1 + \dots + \delta_{i-1}= 1 + \frac{q_1}{p_1} + \dots + \frac{q_1 \dots q_{i-1}}{p_1 \dots p_{i-1}} , \quad i \ge 2.
 \end{equation}
One useful consequence is that for any states $l < i < r$ the probabilities of leaving the interval $(l, r)$ through the right-hand side or through the left-hand side are
\begin{equation}\label{eqn: exit probabilities}
    \Pb_i \{\sigma_l > \sigma_r\} = \frac{u_i - u_l}{u_r - u_l}, \quad \Pb_i \{\sigma_l < \sigma_r\} = \frac{u_r - u_i}{u_r - u_l},
\end{equation}
respectively. 
Next, set by definition
 \begin{equation*}
    r  = \lim_{i \to \infty} u_i 
    =  1 + \sum_{k = 1}^\infty \delta_k 
    = 1 + \sum_{k = 1}^\infty \frac{q_1 \dots q_k}{p_1 \dots p_k}.
\end{equation*}
The Markov chain $(X_k)$ is transient if $r < \infty$ and recurrent if $r = \infty$ (see \cite{D_Y}, IV, \S 4).

\subsection{Preliminaries}\label{preliminaries}
Before we begin the proof of theorems, we need to do some preliminary work. First, introduce the conditional Laplace transform (at the point $1$) of the moment of explosion
    \begin{equation*}
        f_i = \mathsf{E}_i \mathrm{e}^{-\sigma_\infty} = \mathsf{E} [\mathrm{e}^{-\sigma_\infty} \mid X(0) = i ], \quad i \ge 0,
    \end{equation*}
and of the hitting times of states $n \ge 0$
    \begin{equation*}
        f_i^n = \mathsf{E}_i \mathrm{e}^{-\sigma_n} = \mathsf{E} [\mathrm{e}^{-\sigma_n} \mid X(0) = i ], \quad 0 \le i \le n.
    \end{equation*}
\begin{lemma}\label{lemma convergence of f} 
The following properties hold:
\begin{enumerate}[label={\upshape(\arabic*)}]
    \item For a fixed initial state $i\ge 0$, the sequence $\{f_i^n\}_{n \ge i}$ is monotone decreasing;
    \item For a fixed end state $n \in \mN$, the values $f_0^n$, $f_1^n$,\dots, $f_n^n$ are increasing;
    \item For all $i\ge 0$, $\lim_{n \to \infty} f_i^n = f_i$.
\end{enumerate}
\end{lemma}
\begin{proof}
    We immediately get properties (1) and (2) as a consequence of the alternative construction of $X$ (see Remark \ref{alt construction for mom of exp}). Property (3) follows from the Theorem of monotone convergence.
\end{proof}
Let us establish a connection between the Laplace transform and the moment of explosion. 
\begin{lemma}\label{lemma Lap. transf. mom. of exp. link}
Let $P_0$ be the initial distribution of $X$. The following alternative holds:
\begin{align*}
   (a) \lim_{i \to \infty} f_i = 1 &\Longleftrightarrow f_i > 0 \text{ for some } i 
        &\quad (b)\lim_{i \to \infty} f_i = 0 &\Longleftrightarrow f_i = 0 \text{ for some } i \\[-4pt]
    & \Longleftrightarrow f_i > 0 \text{ for all }i 
        && \Longleftrightarrow f_i = 0 \text{ for all }i\\
    & \Longleftrightarrow \sigma_\infty < \infty \ P_0\text{-a.s.}&&\Longleftrightarrow \sigma_\infty = \infty \ P_0\text{-a.s.}
\end{align*}
\end{lemma}
\begin{proof}
From the alternative construction of $X$ (see Remark \ref{alt construction for mom of exp}), we get
    \begin{equation*}
        f_i = \E_i \ee^{-\sigma_\infty} =  \E \ee^{-\sum_{k = i + 1}^\infty \eta_k},
    \end{equation*}
    where $\eta_k$ are independent random variables. Therefore, the statement follows from Kolmogorov's zero-one law.
\end{proof}
Now let us establish conditions that provide $\lim_{i \to \infty} f_i = 1$. To do this, we examine the Laplace transform $f_i^n$ of hitting times more closely. Suppose $n \in \mathbb{N}$ is fixed. Recall that $\{\tau_i^\pm\}$ are the waiting times of the process $X$ from Definition \ref{def of X}, and denote $\varphi_i^\pm = \mathsf{E} \mathrm{e}^{-\tau_i^\pm}$ for $i \ge 0$. Conditioned on $X(0) = n$, $\sigma_n = 0$ almost surely,~hence $f_n^n = 1$. For $X(0) = i \in \{0, \dots, n-1\}$, we can rewrite $\sigma_n = \tau + \tilde{\sigma}_n$, where $\tau$ is the first moment of jump, and $\tilde{\sigma}_n$ is the remaining time before the process hits $n$. By construction of $X$, random variables $\tau$ and $\tilde{\sigma}_n$ are independent. We arrive at the system of equations
\begin{equation}\label{eqn: system for f}
\begin{aligned}
    &f_0^n = \varphi_0^+ f_1^n,\\
    &f_i^n = p_i \varphi_i^+ f_{i+1}^n + q_i \varphi_i^- f_{i-1}^n, \quad  1 \le i \le n-1,\\
    &f_n^n = 1.
\end{aligned}
\end{equation}
This system is difficult to study because of the multiplication by $\varphi_i^\pm$. The idea then is to make it additive in some sense. Denote $F_i^n = 1 - f_i^n$ and $\Phi_i^\pm = 1 - \varphi_i^\pm$. In this notation, the above system transforms as follows:
\begin{equation}\label{eqn: system for F}
\begin{aligned}
    &F_0^n = \Phi_0^+ (1 - F_1^n) + F_1^n,\\
    &F_i^n =  p_i \Phi_i^+ (1 - F_{i+1}^n) + q_i \Phi_i^- (1 - F_{i-1}^n ) + p_i F_{i+1}^n + q_i F_{i-1}^n, \quad  1 \le i \le n-1,\\
    &F_n^n = 0.
\end{aligned}
\end{equation}

Recall a few results about linear recurrent equations and their probabilistic interpretations from the theory of Markov processes on a countable state space. Consider a system
\begin{equation}\label{eqn: system for y}
\begin{aligned}
    &y_0 = A_0 + y_1,\\
    &y_i = A_i  + p_i y_{i+1} + q_i y_{i-1}, \quad 1 \le i \le n-1,\\
    &y_n = 0,
\end{aligned}
\end{equation}
where $A_i$ are some non-negative numbers. Note that \eqref{eqn: system for F} is a particular case of \eqref{eqn: system for y} with $A_i =p_i \Phi_i^+ (1 - F_{i+1}^n) + q_i \Phi_i^- (1 - F_{i-1}^n )$. The above system has a simple probabilistic interpretation: we travel in the state space $\{0,1, \dots, n\}$ with transition probabilities $p_i$, $q_i$ and get a profit of $A_i$ at each state $i$ until we hit $n$. Then $y_i$ is the total expected profit when starting form $i$ upon reaching $n$ (recall that $p_{01}=1$). From this interpretation, a universal lemma for handling such systems follows. Recall that $(X_k)$ is the embedded Markov chain of the process $X$ and denote $\varkappa_n = \inf\{k \ge 0\mid X_k = n\}$. 
\begin{lemma}
\label{lemma solution profit}
    System \eqref{eqn: system for y} has a unique solution given by the formula
    \begin{equation*}
        y_i = \mathsf{E}_i \sum_{k=0}^{\varkappa_n - 1} A_{X_k}, \quad 0 \le i \le n. 
    \end{equation*}
\end{lemma}
We get the following comparison lemma as a corollary.
\begin{lemma}
\label{corollary inequality}
    Suppose $\tilde{y}_i$ satisfy the system \eqref{eqn: system for y} with $\widetilde{A}_i$ instead of $A_i$ and $\widetilde{A}_i \le A_i$ for all $0 \le i \le n$. Then $\tilde{y}_i \le y_i$ for all $0 \le i \le n$.
\end{lemma}
The following lemma from the theory of Markov processes provides a useful explicit formula for the solution of systems of equations we are interested in. 
\begin{lemma}[see \cite{D_Y}, IV, \S 5]
\label{lemma D-Y}
    The solution of system \eqref{eqn: system for y} with coefficients $A = \{A_i\}_{i=0}^n$ is given by $y_j  = S_n (A) - S_j(A)$, where
    \begin{equation*}
        S_j(A) =  \sum\limits_{k=0}^{j-1} \sum\limits_{i=0}^{k} A_i \frac{q_{i+1} \dots q_k}{p_i \dots p_k}, \quad 0 \le j \le n. 
    \end{equation*}
\end{lemma}

\subsection{Proof of the explosion condition (Theorem \ref{explosion theorem})}\label{regularity proof}
We first assume the right-hand side of \eqref{eqn: explosion condition} holds, that is,
\begin{equation*}
    \sum_{k = 0}^\infty  \left(\sum_{i=0}^{k} \nu_i \right) \delta_k = \sum_{k = 0}^\infty  \sum_{i=0}^{k} (1-\E \ee^{-\tau_i}) \frac{q_{i+1} \dots q_k}{p_i \dots p_k} < \infty.
\end{equation*}
From this we prove that $\mathsf P \{X \text{ explodes to }\infty \} = 1$. Put by definition
\begin{equation*}
    \Phi_i \coloneqq 1-\E \ee^{-\tau_i} = 1 - p_i\E \ee^{-\tau_i^+}-q_i\E \ee^{-\tau_i^-} = p_i \Phi_i^+ + q_i \Phi_i^-.
\end{equation*}
Then, following Lemma \ref{lemma D-Y}, we can denote the above series by $S_\infty (\Phi) = \lim_{n \to \infty} S_n (\Phi)$, and hence our assumption becomes $S_\infty (\Phi) < \infty$. 

Recall that $F_i^n = 1 - f_i^n = 1 - \E_i \ee^{-\sigma_n}$ satisfy system \eqref{eqn: system for F}. We want to compare $F_i^n$ with a solution to a simpler system. Note that, for all indices,
\begin{equation}\label{eqn: estimate with Phi i}
p_i \Phi_i^+ (1 - F_{i+1}^n) + q_i \Phi_i^- (1 - F_{i-1}^n ) \le p_i \Phi_i^+ + q_i \Phi_i^- = \Phi_i.
\end{equation}
Let $G_i^n$, $0 \le i \le n$, satisfy the system
\begin{equation}\label{eqn: system for G}
\begin{aligned}
    &G_0^n = \Phi_0 + G_1^n,\\
    &G_i^n = \Phi_i  + p_i G_{i+1}^n + q_i G_{i-1}^n, \quad  1 \le i \le n-1,\\
    &G_n^n = 0.
\end{aligned}
\end{equation}
By Lemma \ref{lemma solution profit}, $G_i^n$ are determined uniquely. Moreover, $F_i^n \le G_i^n$ for $0 \le i \le n$ due to estimate~\eqref{eqn: estimate with Phi i} and comparison Lemma \ref{corollary inequality}. Combining this with Lemma \ref{lemma D-Y}, we get
\begin{equation*}
    F_i^n \le G_i^n = S_n (\Phi) - S_i (\Phi), \quad 0 \le i \le n.
\end{equation*}
By Lemma \ref{lemma convergence of f}, $\lim_{n \to \infty} F_i^n =  \lim_{n \to \infty} (1-f_i^n) = 1 - f_i$, hence we can pass to the limit
\begin{equation*}
1 - f_i = \lim_{n \to \infty} F_i^n \le S_\infty (\Phi) - S_i (\Phi).
\end{equation*}
Since $S_\infty (\Phi) < \infty$ and $1 - f_i \ge 0$, it follows that 
\begin{equation*}
    0 \le \lim_{i \to \infty} (1-f_i) \le \lim_{i \to \infty} (S_\infty (\Phi) - S_i (\Phi)) = 0.
\end{equation*}
Hence $\lim_{i \to \infty} f_i = 1$, and Lemma \ref{lemma Lap. transf. mom. of exp. link} implies that $\mathsf P \{X \text{ explodes to }\infty\} = 1$.

Let us show that $\mathsf P \{X \text{ explodes to }\infty  \} = 1$ implies $S_\infty (\Phi) < \infty$. The idea is to estimate $F_i^n = 1 - \mathsf E_i \mathrm{e}^{-\sigma_n}$ from below, using the fact that it satisfies the system \eqref{eqn: system for F}. Lemma \ref{lemma convergence of f} implies that, for all $0 \le i \le n$,
\begin{equation*}
   1- F_i^n = f_i^n = \mathsf E_i \mathrm{e}^{-\sigma_n} \ge \mathsf E_i \mathrm{e}^{-\sigma_\infty} \ge \mathsf E_0 \mathrm{e}^{-\sigma_\infty} = f_0 > 0.
\end{equation*}
Denoting $c = f_0$, we have 
\begin{equation}\label{eqn: Phi estimate from below}
     p_i \Phi_i^+ (1 - F_{i+1}^n) + q_i \Phi_i^- (1 - F_{i-1}^n ) \ge p_i \Phi_i^+ c + q_i \Phi_i^- c = \Phi_i c.
\end{equation}
Now consider a system
\begin{equation}\label{eqn: system for H}
\begin{aligned}
    &H_0^n = \Phi_0 c + H_1^n,&\\
    &H_i^n = \Phi_i c + p_i H_{i+1}^n + q_i H_{i-1}^n, \quad & 1 \le i \le n-1,\\
    &H_n^n = 0.&
\end{aligned}
\end{equation}
Then from \eqref{eqn: Phi estimate from below} and Lemmas \ref{lemma solution profit}, \ref{corollary inequality}, it follows that $H_i^n \le F_i^n$ for all $0 \le i \le n$. On the other hand, from Lemma \ref{lemma D-Y}, we get
\begin{equation*}
H_i^n = c G_i^n = c (S_n (\Phi) - S_i (\Phi)).
\end{equation*}
This leads us to the estimate
\begin{equation*}
    S_n (\Phi) - S_i (\Phi) = \frac{1}{c} H_i^n \le \frac{1}{c} F_i^n = \frac{1}{c} (1 - \mathsf E_i \mathrm{e}^{-\sigma_n}) \le \frac{1}{c}
\end{equation*}
for all $0 \le i \le n$, which implies the convergence of $S_n (\Phi)$ to a finite limit. 

The statement \eqref{eqn: no explosion condition} for non-explosion follows from \eqref{eqn: explosion condition} and Lemma \ref{lemma Lap. transf. mom. of exp. link}. This finishes the proof of Theorem \ref{explosion theorem}.

\subsection{Proof of the implosion condition (Theorem \ref{implosion theorem})}\label{implosion proof}
Assume the right-hand side of \eqref{eqn: implosion condition} holds, that is,
\begin{equation*}
    \sum_{k = 0}^\infty \left(\sum_{i=k+1}^\infty \nu_i \right)\delta_k  < \infty \quad\text{and} \quad \sum_{k=0}^\infty \delta_k = \infty.
\end{equation*}
From this we prove that $\Pb \{X \text{ implodes from }\infty\} = 1$, i.e., $X$ comes down from $\infty$ to any state $l$ in a finite time. We first examine the hitting time $\sigma_l =\inf\{t \ge 0 \mid X(t) = l \}$ when starting from some finite state $i \ge l$. Add a point $r$ on the right of $i$, and consider the minimum of hitting times $\sigma_l$ and $\sigma_r$, denoted by $\sigma_l \wedge \sigma_r$. We define 
\begin{equation*}
    F_i^{l, r} = 1 - \E_i \ee^{- \sigma_l \wedge \sigma_r}, \quad l \le i \le r,
\end{equation*}
and, analogously to derivation of \eqref{eqn: system for f} and \eqref{eqn: system for F}, arrive at the system
\begin{equation}\label{eqn: system for F^lr}
\begin{aligned}
    &F_l^{l, r} = F_r^{l, r} = 0,\\
    &F_i^{l, r} = p_i \Phi_i^+ (1 -  F_{i+1}^{l, r})  + q_i \Phi_i^-  (1 - F_{i-1}^{l, r} ) + p_i F_{i+1}^{l, r} + q_i F_{i-1}^{l, r}, \quad  l < i < r.
\end{aligned}
\end{equation}
We want to compare the solution of \eqref{eqn: system for F^lr} with a solution of a simpler system
\begin{equation}\label{eqn: system for G^lr}
\begin{aligned}
    &G_l^{l, r} = G_r^{l, r} = 0,\\
    &G_i^{l, r} = \Phi_i + p_i G_{i+1}^{l, r} + q_i G_{i-1}^{l, r}, \quad  l < i < r.
\end{aligned}
\end{equation}
It has the same probabilistic interpretation as \eqref{eqn: system for y}, but now $G_i^{l, r}$ is the total expected profit in interval $(l, r)$ when starting from $i$ up to hitting $l$ or $r$. Recall that $(X_k)$ is the embedded Markov chain of the process $X$, and $\varkappa_n = \inf\{k \ge 0\mid X_k = n\}$. We arrive at the formula
\begin{equation*}
    G_i^{l,r} = \E_i \sum_{k = 0}^{\varkappa_l \wedge \varkappa_r - 1} \Phi_{X_k}, \quad l \le i \le r.
\end{equation*}
Using a suitable version of comparison Lemma \ref{corollary inequality}, we get that $F_i^{l,r} \le G_i^{l,r}$ for all $l \le i \le r$. 

Since $\sum_{k=0}^\infty \delta_k = \infty$, the Markov chain $(X_k)$ is recurrent due to Section \ref{transience, recc and scale}. Recall that in the recurrent case $\sigma_\infty = \infty$ almost surely,\,implying $\lim_{r \to \infty} \sigma_l \wedge \sigma_r = \sigma_l$ almost surely. Hence by Theorem of monotone convergence there exist limits
\begin{equation*}
    F_i^l \coloneqq \lim_{r \to \infty} F_i^{l,r} = 1 - \E_i \ee^{- \sigma_l }, \quad
    G_i^l \coloneqq  \lim_{r \to \infty}  G_i^{l,r} = \E_i \sum_{k = 0}^{\varkappa_{l} - 1} \Phi_{X_k}
\end{equation*}
with $F_i^l \le G_i^l$ for all $i \ge l$. Let us now rewrite the total profit as a sum over all states
\begin{equation}\label{eqn: first formula for G i l}
    G_i^l 
    =  \E_i \sum_{k = 0}^{\varkappa_{l} - 1} \Phi_{X_k}
    =\sum_{j = l+1}^\infty \E_i \sum_{k = 0}^{\varkappa_{l}-1} \Phi_j \1_{\{X_k = j\}} 
    \eqqcolon \sum_{j = l+1}^\infty y_i (j),
\end{equation}
where the value $y_i (j)$ resembles the average profit at the state $j > l$ when starting from $i \ge l$ up to hitting $l$. Let us reduce it to the case of $j=i$. If the states $i$ and $j$ are different, the process gets to $j$ earlier than to $l$ with probability $\Pb_i \{\sigma_j < \sigma_l\}$, and starts afresh from $j$. Hence we have
\begin{equation*}
    y_i(j) = y_j (j) \Pb_i \{\sigma_j < \sigma_l\}.
\end{equation*}
Let us denote $y_j = y_j(j )$ and examine it more closely. Starting from a state $j$, we first get a profit of $\Phi_j$. Then we either go right with probability $p_j$ and return to $j$ with probability one due to recurrence, or go left with probability $q_j$ and return to $j$ with probability $\Pb_{j-1} \{\sigma_j < \sigma_l\}$. This leads to the relation
\begin{equation*}
    y_j = \Phi_j + p_j y_j + q_j y_j \Pb_{j-1} \{\sigma_j < \sigma_l\},
\end{equation*}
from which we get the explicit formula
\begin{equation}\label{eqn: first formula for y}
    y_j = \frac{\Phi_j}{1 - p_j - q_j \Pb_{j-1} \{\sigma_j < \sigma_l\}} =\frac{\Phi_j}{q_j \Pb_{j-1} \{\sigma_j > \sigma_l\}}.
\end{equation}
The latter probability can be calculated as follows (see Section \ref{transience, recc and scale}):
\begin{equation*}
    \Pb_{j-1} \{\sigma_j > \sigma_l\} = \frac{u_j - u_{j-1}}{u_j - u_l} .
\end{equation*}
Recall that $\nu_j = \Phi_j \frac{p_1 \dots p_{j-1}}{q_1 \dots q_j}$. Hence \eqref{eqn: first formula for y} becomes
\begin{equation*}
    y_j 
    = \frac{\Phi_j (u_j - u_l)}{q_j(u_j - u_{j-1})} 
    = \frac{\Phi_j (u_j - u_l)}{q_j\delta_{j-1}} 
    = \nu_j (u_j - u_l) 
    = \nu_j \sum_{k=l}^{j-1} \delta_k.
\end{equation*}
Let us return to the total profit $G_i^l$ in the interval $(l, \infty)$ up to hitting $l$. We can expand \eqref{eqn: first formula for G i l}:
\begin{align*}
    G_i^l 
    = \sum_{j = l+1}^\infty y_j \Pb_i\{\sigma_j < \sigma_l\} \le \sum_{j = l+1}^{\infty} y_j  
    = \sum_{j = l+1}^{\infty} \nu_j \sum_{k = l}^{j-1} \delta_k 
    = \sum_{k = l}^{\infty} \left(  \sum_{j = k+1}^{\infty} \nu_j \right)\delta_k < \infty.
\end{align*}
Moreover, since $\Pb_i\{\sigma_j < \sigma_l\}$ is increasing in $i$, the monotone convergence theorem implies 
\begin{equation}\label{eqn: final formula for G_l}
   G^l 
   \coloneqq \lim_{i \to \infty} G_i^l 
   = \lim_{i \to \infty} \sum_{j = l+1}^\infty y_j \Pb_i\{\sigma_j < \sigma_l\} 
   = \sum_{j = l+1}^{\infty} y_j  
   = \sum_{k = l}^{\infty} \left(  \sum_{j = k+1}^{\infty} \nu_j \right)\delta_k.
\end{equation}
Recall the alternative construction of $X$ (Section \ref{alt construction}), and the alternative definition of implosion (Definition \ref{alt def of implosion}), in which $\rho_\infty = \sum_{k=1}^\infty \theta_k$ represents the first time of hitting $0$ when starting at $\infty$. Observe that $F_i^l = 1 - \E_i \ee^{- \sigma_l } =1 - \E \ee^{-\sum_{k = l+1}^i \theta_k}$. Then there exists a limit $F^l = \lim_{i\to \infty} F_i^l$, and since $F_i^l \le G_i^l$ for all $i \ge l$, we have
\begin{equation}\label{eqn: F_l le G_l}
    1-\E \ee^{-\sum_{k = l+1}^\infty \theta_k} 
    = F^l \le G^l 
    = \sum_{k = l}^{\infty} \left( \sum_{j = k+1}^{\infty} \nu_j \right)\delta_k.
\end{equation}
The right hand side of  \eqref{eqn: F_l le G_l} goes to $0$ as $l \to \infty$, because the series is convergent. This implies that $\E \ee^{-\sum_{k = l+1}^\infty \theta_k} > 0$ for large enough $l$, and hence for all $l$. Thus, $\rho_\infty = \sum_{k = 1}^\infty \theta_k < \infty$ with probability one due to Kolmogorov's zero-one law. This finishes the first part of the proof. 

Suppose $\rho_\infty < \infty$ almost surely. Then $\sum_{k=0}^\infty \delta_k = \infty$ due to Section \ref{transience, recc and scale}. Let us show that
\begin{equation*}
    \sum_{k = 0}^\infty \left(\sum_{i=k+1}^\infty \nu_i \right)\delta_k  < \infty.
\end{equation*}
From the assumption, it is clear that for all $l \ge 0$
\begin{equation*}
    f^l \coloneqq \E \ee^{-\sum_{k = l+1}^\infty \theta_k} \ge \E \ee^{-\sum_{k = 1}^\infty \theta_k} = f^0 > 0.
\end{equation*}Analogously to explosion condition proof, using \eqref{eqn: final formula for G_l}, we obtain the estimate
\begin{equation*}
    \sum_{k = l}^\infty \left(\sum_{i=k+1}^\infty \nu_i \right)\delta_k  = G^l \le \frac{1}{f^0} F^l \le \frac{1}{f^0},
\end{equation*}
which implies the convergence of the series and finishes the proof of Theorem \ref{implosion theorem}.

\section*{Acknowledgements}
\phantom{ }
\vspace{-1em}

A. Pilipenko  thanks  the Swiss National Science
Foundation for partial support  of the paper (grants No. IZRIZ0\_226875, No. 200020\-\_200400, No. 200020\_192129) and also  the Isaac Newton Institute for Mathematical Sciences, Cambridge, for support and hospitality during the programme \emph{Stochastic systems for anomalous diffusion}. This work was supported by EPSRC grant EP/Z000580/1. 

This work was supported by a grant from the Simons Foundation (SFI-PD-Ukraine-00014586, V.T.)

\bibliographystyle{plain}

\end{document}